\documentclass[letterpaper, 10pt, conference, english]{ieeeconf}

\IEEEoverridecommandlockouts                              % This command is only needed if 
\usepackage{graphicx}
\usepackage{amsmath}
\usepackage{mathtools, amssymb}
\usepackage{varioref}
\labelformat{equation}{(#1)}
\usepackage{amsfonts}

\usepackage{bbm,dsfont}
\usepackage{color}
\usepackage{enumerate}

\usepackage{cite}
\usepackage{amsthm}
\theoremstyle{plain}
\newtheorem{thm}{Theorem}

\newtheorem{claim}{Claim}
\theoremstyle{definition}
\newtheorem{asmptn}{Assumption}
\newtheorem{remark}{Remark}

\newtheorem{example}{Example}

\usepackage{dsfont}

\renewcommand{\epsilon}{\varepsilon}

\usepackage{algorithm}
\usepackage{algpseudocode}

\usepackage{epstopdf}

\begin{document}

\title{\LARGE \bf
Value of forecasts in planning under uncertainty:\\
Extended version
}

\author{Konstantinos Gatsis, Ufuk Topcu, and George J. Pappas
\thanks{This is an extended version of the paper to appear at the 2015 American Control Conference, containing the proofs.}%
\thanks{
This work was supported in part by NSF CNS award number 1312390, and by TerraSwarm, one of six centers of STARnet, a Semiconductor Research Corporation program sponsored by MARCO and DARPA. The authors are with the Department of Electrical and Systems Engineering, University of Pennsylvania, 200 South 33rd Street, Philadelphia, PA 19104. Email: \{kgatsis, utopcu, pappasg\}@seas.upenn.edu. }}

\maketitle

%%% Add these to force page numbers
%\thispagestyle{plain}
%\pagestyle{plain}
%% Replace them with the following:
\thispagestyle{empty}
\pagestyle{empty}

\begin{abstract}
In environments with increasing uncertainty, such as smart grid applications based on renewable energy, planning can benefit from incorporating forecasts about the uncertainty and from systematically evaluating the utility of the forecast information. We consider these issues in a planning framework in which forecasts are interpreted as constraints on the possible probability distributions that the uncertain quantity of interest may have. The planning goal is to robustly maximize the expected value of a given utility function, integrated with respect to the worst-case distribution consistent with the forecasts. Under mild technical assumptions we show that the problem can be reformulated into convex optimization. We exploit this reformulation to evaluate how informative the forecasts are in determining the optimal planning decision, as well as to guide how forecasts can be appropriately refined to obtain higher utility values. A numerical example of wind energy trading in electricity markets illustrates our results.
\end{abstract}

\section{Introduction}

The intrinsically uncertain nature of renewable energy sources, for example their dependence on local weather conditions~\cite{mills2011}, impedes the reliable operation of the electricity grid~\cite{makarov2009}. Forecasts about the renewable energy generation provide a means to cope with this uncertainty. Hence the successful incorporation of forecasts into planning grid operation emerges as an important challenge, as well as the problem of obtaining forecasts that provide the most valuable information for planning.

%Integrating renewable energy sources in the electricity grid is challenging due to the intrinsically uncertain nature of such sources. Generation based on wind or solar power is volatile due to varying local weather conditions~\cite{mills2011}, or other factors, e.g., the lack of solar resources during night. This variability impedes reliable grid operation~\cite{makarov2009}, for example meeting the consumers' energy demands and maintaining system stability. To cope with such variability, forecasts of renewable energy generation are taken into account during planning and operation of the grid. As a result, successfully incorporating forecasts in the decision-making process becomes increasingly crucial.

%'The successful use of wind power forecasts by system operators can significantly reduce integration challenges and costs.'~\cite[p. 65]{Wind_report}

Uncertainty is usually considered in a stochastic framework during the grid planning stage. Generation, for example, is modeled as a random variable, and planning decisions are made a certain length of time before the generation is realized at operation time. Planning in this context can provide probabilistic guarantees, for example, that undesirable events will happen with low probability, or that expected operation costs/utilities are optimized. Examples can be found in, e.g., dispatch problems~\cite{risk_limiting}, unit commitment~\cite{stochastic_UC}, and participation in electricity markets~\cite{BitarEtal, Pinson_Trading}. 

Solving such stochastic optimization problems relies on the availability of the underlying probability distribution of the uncertain quantities. The renewable generation forecasts however do not provide complete descriptions of the underlying uncertainty, in the sense that they do not completely describe the probability distribution of the generation. Certain questions arise in this context:
\begin{enumerate}
\item How can forecasts of the quantities of interest be incorporated in a stochastic planning framework?
\item How informative is a set of given forecasts to a specific planning problem?%, and how can forecasts be refined to yield a higher relative objective value?
\end{enumerate}
The first question is crucial to reliably integrate renewable generation in the grid. The second question becomes practically relevant if obtaining forecasts during the planning stage incurs significant cost. This could be, for example, the computational cost of running complex renewable generation forecasting algorithms, typically combining numerical weather prediction models, historical data, local measurements,  and simulations%, and conversions of weather quantities to power, e.g., wind speed to power generation
~\cite{wind_forecast}. %Hence knowing how to obtain more informative forecasts could reduce the associated computational cost of running such algorithms during the planning stage. 
Alternatively these costs could be monetary, as is the case when grid operators and power producers do not create their own forecasts but instead purchase them in the form of products from companies specializing in forecasting~\cite{NCAR, Forecast_industries}. %
Understanding how valuable are the forecasts in determining planning decisions could help reduce the associated costs of obtaining such forecasts.
%It is also worth noting that evaluating how informative forecasts are depends both on what the particular objectives of the planning are and on the particular forecast values provided.

In this paper we take a step towards mathematically formalizing and answering the above questions. The simplest type of forecast for an uncertain quantity that is modeled as random is a point forecast, i.e., a single value thought as the most likely or the expected value of the quantity. However, more sophisticated types of forecasts, such as probabilistic forecasts, have been introduced for the case of renewable generation~\cite{wind_forecast, Pinson_Trading}. Unlike point forecasts, probabilistic forecasts provide higher fidelity information about the probability distribution of the random quantity. Prediction intervals, i.e., bounds on the probability that the quantity takes values in certain intervals, are a common example~\cite{wind_forecast}.

We consider a general stochastic planning framework under probabilistic forecasts (Section~\ref{sec:problem}). We look for a planning decision that maximizes the expected value of a given utility function, but the probability distribution of the uncertain quantity, with respect to which the expectation is computed, is unknown and only described via the forecasts. We interpret the probabilistic forecasts as constraints defining a set of possible probability distributions for which the decision needs to account. Consequently a robust (max-min) planning problem is formulated to determine the decision that maximizes the worst-case expected utility, where the worst-case expectation is selected from the set of probability distributions consistent with the forecasts.

In Section~\ref{sec:solution} we utilize Lagrange duality theory~\cite{LP_over_distributions} to show that the robust planning problem admits an equivalent reformulation to a single maximization problem with the introduction of auxiliary (dual) variables. This resulting problem is convex under certain relatively mild assumptions on the planning utility function, but it includes an infinite number of constraints. For the special case of forecasts in the form of prediction intervals we show that the constraints can be reduced to a finite number (Section~\ref{sec:planning_prediction_intervals}), resulting in a standard (finite-dimensional) convex optimization problem that can be solved readily~\cite{Boyd_convex}.

The problem of selecting the worst-case probability distribution given probabilistic descriptions has been discussed previously in the context of optimal uncertainty quantification~\cite{Han_uncertainty_quantification}. The difference in our formulation however is that an additional planning optimization level is considered. Related planning formulations under uncertainty about probability distributions have been considered in, e.g.,~\cite{Distributionally_robust, Ambiguous_chance_constrained} and references therein, where the optimal planning is solved in a data-driven fashion based on samples from the unknown underlying distribution. Conceptually similar approaches have been considered in the context of power systems, where expected values in stochastic optimization are approximated using samples obtained from wind power forecasting algorithms~\cite{stochastic_UC, probabilistic_forecast_application}. In contrast, planning in our work is decoupled from the data collection and does not rely on sample-based approximations, but instead utilizes the output of the forecasting procedures, i.e., the probabilistic forecasts themselves. 

A further novelty of our approach is that it allows to characterize how valuable is the forecast information in determining the optimal planning decision. This characterization follows from a sensitivity analysis of the optimal planning objective value with respect to the forecast constraints (Section~\ref{sec:forecast_updates}). %, and is captured in the optimal dual variables introduced in the problem reformulation of Section~\ref{sec:solution}. %Intuitively, the larger the dual value is for some forecast, the larger utility function improvement we get if we can refine this forecast. 
Moreover, under the hypothesis that forecasts can be refined by, e.g., a forecasting oracle, we develop a sensitivity-driven procedure that sequentially selects forecast refinements to yield a higher robust planning utility value. Numerical examples of the proposed approach for trading wind energy in electricity markets~\cite{BitarEtal, Pinson_Trading} are presented in Section~\ref{sec:numerical}. Finally, we conclude with a discussion on our contributions and future work in Section~\ref{sec:remarks}.

%As a motivating example throughout the paper we consider the problem of trading wind energy in a day-ahead electricity market, considered in, e.g.,~\cite{BitarEtal, Pinson_Trading}. A power producer based on wind energy receives forecasts of future generation and needs to bid a power amount in the market to maximize its profit. Numerical examples of the proposed robust optimization-based approach to bidding are shown in Section~\ref{sec:numerical}. We also illustrate how our methodology to refine forecasts can be employed to improve the bidding profits. Finally, we conclude with a discussion on our contributions and future work in Section~\ref{sec:remarks}.

\section{Problem description}\label{sec:problem}

We consider a decision-making framework under uncertainty. The unknown quantity of interest, for example the unknown renewable power generation, is denoted by $x$ and we assume that $x$ takes real values in a subset $\mathcal{X} \subseteq \mathbb{R}$ of the reals. We model $x$ as a random variable following some probability distribution on $\mathcal{X}$. As we will make clear, this distribution is not known completely. Hence there is uncertainty about the distribution of the random quantity of interest.

Suppose a decision-maker needs to select the value of a controlled parameter $b$ before the random quantity $x$ is realized, i.e., drawn from its probability distribution. We assume $b$ takes values in a convex subset $\mathcal{B} \subseteq \mathbb{R}$ of the reals. To rank the different choices for $b$ we assume that a function $J(x,b)$ models the utility to the user if decision $b$ is taken a priori and the unknown quantity takes value $x$. Technically we assume the utility function is continuous in both variables $x$ and $b$, and also concave in $b$ for every value of $x \in \mathcal{X}$, and concave in $x$ for every value of $b \in \mathcal{B}$.

\begin{example}\label{example:bidding}
Consider the case of a wind power plant participating in a day-ahead electricity market~\cite{BitarEtal, Pinson_Trading}. The producer provides a capacity bid in the market, which we can denote by a decision variable $b$, representing the estimated renewable power generation that will be supplied to the grid at a future time interval. We let $b$ take values in the unit interval, $b \in [0,1]$, which can be interpreted as a normalized percentage with respect to the maximum capacity of the generator. If we denote by $x$ the actual realized power, also normalized in $x \in [0,1]$, a simple model~\cite{BitarEtal} of the producer's monetary profit from a realization $x$ after bidding $b$ is
\begin{equation}\label{eq:J_bidding}
	J(x,b) = p\,b - q [b-x]_+ ,
\end{equation}
where $[\theta]_+ = \max\{\theta,0\}$. The constant $p>0$ rewards high bids, while the constant $q>0$ penalizes a shortfall of generation compared to the bid, i.e., when $b>x$. It is assumed that $q>p$, implying that the profit decreases for higher shortfalls. The utility function, as can be seen in Fig. \ref{fig:concavity}, satisfies the convexity assumptions of the general problem formulation.
\end{example}

\begin{figure}[!t]
  \centering
  \includegraphics[width=1.0\columnwidth]{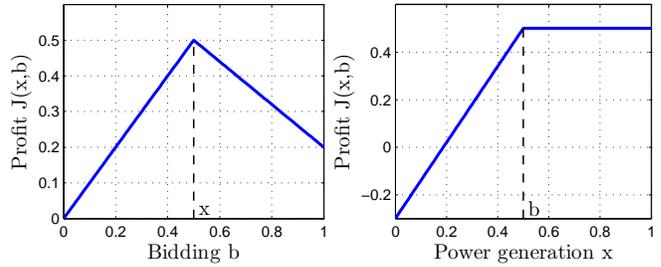}
  \caption{Profit $J(x,b)$ from bidding in electricity markets for $p=1$, $q=1.6$. On the left generation is kept constant $x = 0.5$ while bid $b$ varies. On the right bid is kept constant $b = 0.5$ while generation $x$ varies.}
  \label{fig:concavity}
\end{figure}

A common assumption in planning problems under uncertainty is that the underlying probability distribution of the unknown quantity is known. In particular, if the random variable $x$ has a probability distribution $F$, the expected utility to the user from a choice $b$ is given by
\begin{equation}\label{eq:expected_profit}
	\mathbb{E}_{x \sim F} \, J(x,b) = \int_x \, J(x,b) \, dF(x),
\end{equation}
where the integral is over the range of values $x \in \mathcal{X}$. The value $b$ that maximizes the expected utility in \ref{eq:expected_profit} becomes the optimal decision in this case. 

In this paper however the underlying distribution $F$ is not completely known but only forecasts about the random quantity $x$ are available. We consider probabilistic forecasts~\cite{wind_forecast, Pinson_Trading}, which provide information about the underlying probability distribution $F$ of $x$. More specifically we consider given (measurable) functions $g_i(x)$, for $i=1,\ldots,n$, of the random quantity $x$, and forecasts stating that the expected value of these functions is upper bounded by some parameters $\epsilon_i$, i.e.,
\begin{equation}\label{eq:forecast_inequality}
	\mathbb{E}_{x \sim F} \, g_i(x) \leq \, \epsilon_i, \; i=1,\ldots,n .
\end{equation}
Both the functions $g_i(x)$ and the parameters $\epsilon_i$ in these inequalities are given to the decision-maker, e.g., provided by some forecasting algorithm. We interpret the forecasts \ref{eq:forecast_inequality} as a set of constraints on the unknown underlying distribution $F$, narrowing the uncertainty of the decision-maker regarding the distribution of $x$. Based on this interpretation, we will interchangeably use the terms forecasts and constraints when referring to \ref{eq:forecast_inequality}. 

The constraint-based characterization of the uncertainty in \ref{eq:forecast_inequality} matches many common types of forecasts. For example bounds on mean value, second moment, or higher-order moments can be expressed in the form of \ref{eq:forecast_inequality} by appropriately selecting the function $g_i(x)$. In the following example we detail another special case, the prediction intervals. We will revisit this case later.

\begin{example}\label{example:prediction_intervals}
An example motivated by recent forecast methods for renewable generation~\cite{Pinson_Trading} are prediction intervals. Suppose the renewable generation $x$ takes values in a normalized interval $\mathcal{X} = [0,1]$, partitioned into $m$ sub-intervals $[x_{i-1}, x_i)$ for $i=1,\ldots,m$, where $x_0 = 0$, $x_m = 1$, and the forecasts take the form
\begin{equation}\label{eq:prediction_intervals}
	\underline{\delta}_i  \leq \mathbb{P}(x_{i-1} \leq x < x_{i} ) \leq \overline{\delta}_i, \quad i=1, \ldots, m.
\end{equation}
In other words, this forecast provides upper and lower bounds on the probability that $x$ takes value in each of the intervals. These forecasts can be reformulated into the form of \ref{eq:forecast_inequality} by defining the indicator functions 
\begin{equation}
	g_i (x) = \mathds{1}\{[ x_{i-1}, x_{i}) \} = \left\{ 
	\begin{array}{ll}
	1 &\text{if } x_{i-1} \leq x < x_{i}, \\
	0 &\text{otherwise},
	\end{array}
	\right.
\end{equation}
for $i = 1,\ldots,m$. Then the constraints \ref{eq:prediction_intervals} are equivalently written in the form of \ref{eq:forecast_inequality} as
\begin{equation}\label{eq:standard_prediction_intervals}
	\mathbb{E}_{x \sim F} \, g_i(x) \leq \, \overline{\delta}_i , \; \text{and} \quad 
	\mathbb{E}_{x \sim F} \, - g_i(x) \leq \, - \underline{\delta}_i.
\end{equation}
A pictorial representation is given in Fig.~\ref{fig:prediction_intervals}.
\end{example}

\begin{figure}[!t]
  \centering
  \includegraphics[width=1.00\columnwidth]{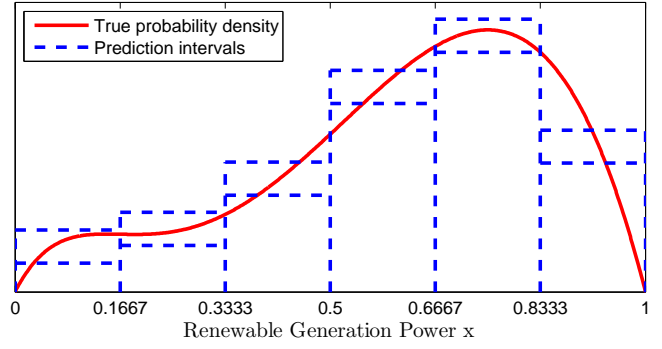}
  \caption{Pictorial representation of prediction intervals (Example~\ref{example:prediction_intervals}). The true but unknown probability density function of the renewable generation is plotted. Forecasts of the generation are available in the form of prediction intervals, which are upper and lower bounds on the real probability mass placed on $m = 6$ equal intervals of generation amounts. For illustration purposes the prediction intervals are shown by blocks, with the areas under the blocks equal to the upper and lower bounds, $\overline{\delta}_i$ and $\underline{\delta}_i$ respectively, according to \ref{eq:prediction_intervals}.}
  \label{fig:prediction_intervals}
\end{figure}

We make the following technical assumption about the forecasts provided in \ref{eq:forecast_inequality}.

\begin{asmptn}\label{as:strict_inequality}
There exists a probability distribution $F$ on $\mathcal{X}$ such that the forecasts \ref{eq:forecast_inequality} are satisfied with strict inequality, i.e., $\mathbb{E}_{x \sim F} \, g_i(x) < \epsilon_i$ for all $i=1,\ldots,n$.
\end{asmptn}

This assumption states that the constraints \ref{eq:forecast_inequality} are feasible, and furthermore that strict feasibility holds. To motivate the feasibility, we can think of the actual underlying distribution of the random quantity, which is not known to the decision-maker, as satisfying the constraints \ref{eq:forecast_inequality}. The strict feasibility is assumed for certain technical reasons in order to establish the results in the sequel of this paper, but is not practically restrictive. The parameters $\epsilon_i$ can always be perturbed to make the strict inequality assumption hold.

Given the constraint-based interpretation of the forecasts in \ref{eq:forecast_inequality}, we can alternatively describe the planner's uncertainty concerning the true distribution $F$ of the unknown quantity $x$ by a set $\mathcal{F}$ containing the (possibly uncountably many) distributions that satisfy the forecasts \ref{eq:forecast_inequality}. The set of all distributions consistent with forecasts \ref{eq:forecast_inequality} is
\begin{equation}\label{eq:set_F_defn}
	\mathcal{F}(\epsilon) = \{ F \in \mathcal{D}_{\mathcal{X}} \,: \mathbb{E}_{x \sim F} \, g_i(x) \leq \, \epsilon_i, \; i=1, \ldots, n \},
\end{equation}
where we denote the set of all probability distributions on $\mathcal{X}$ by $\mathcal{D}_{\mathcal{X}}$. Here we parametrize the set with the forecast bounds $\epsilon_i$, $i=1,\ldots,n$, grouped in a vector $\epsilon$. For the current problem development, $\epsilon$ is fixed, but we introduce this parametrization for later use in Section~\ref{sec:forecast_updates}. By Assumption~\ref{as:strict_inequality} the set $\mathcal{F}(\epsilon)$ has a nontrivial interior, e.g., the true underlying distribution belongs in the set.

Having defined the utility $J(x,b)$ of different choices $b$ to the decision-maker, as well as the possible distributions $F \in \mathcal{F}(\varepsilon)$ of the random quantity $x$ that the decision-maker can anticipate according to the probabilistic forecasts, we now pose a planning problem. The decision-maker needs to determine the value $b$ that robustly maximizes the worst-case expected utility anticipated from the forecasts, mathematically captured as a robust optimization problem
\begin{equation}\label{eq:robust_planning}
	P^*(\epsilon) = \, \underset{b \in \mathcal{B}}{\text{maximize}} \; \inf_{F \in \mathcal{F}(\epsilon)} \; \mathbb{E}_{x \sim F} \, J(x,b). 
\end{equation}
In other words, the planner looks for a decision that leads to the most favorable objective value assuming the worst-case distribution among the ones consistent with the forecast $\mathcal{F}(\epsilon)$. We denote this robustly optimal objective value by $P^*(\epsilon)$, again parametrized by the forecast bounds $\epsilon$ for later use. We also denote the optimal planning decision as $b^*(\epsilon)$, assuming it exists.

The difficulty in solving the robust planning problem \ref{eq:robust_planning} lies in the max-min formulation. For every choice $b$, one needs to solve a minimization problem with respect to the distribution $F$ to evaluate how good the choice is, and then infer how to improve on $b$ to maximize the objective function. To overcome this complexity, in the following section, we examine how the robust planning problem \ref{eq:robust_planning} can be equivalently written as a single-layer optimization problem. For the special case of forecasts given by prediction intervals (Example \ref{example:prediction_intervals}), we obtain an equivalent convex optimization problem from which $b^*(\epsilon)$ can be determined efficiently. 

We proceed in Section~\ref{sec:forecast_updates} to characterize how informative the $n$ forecasts given by \ref{eq:forecast_inequality} are in determining the optimal planning decisions. To this end, we examine how sensitive the robust planning objective value $P^*(\epsilon)$ is to changes in the forecast parameters $\epsilon$. Based on this analysis, we also develop a methodology for improving the planning utility by appropriately refining forecasts.

%Simultaneously\KG{rewrite} we perform a sensitivity analysis with respect to forecast parameters. More explicitly, given the parametric representation of the forecast $\mathcal{F}(\varepsilon)$, we examine the sensitivity of the optimal planning objective value to changes in the parameters $\varepsilon$. We then develop a methodology for updating forecasts in order to improve planning decisions. We perform this by exploiting the sensitivity analysis as a measure of how valuable are the specific forecasts to the planning problem.

\section{Robustly optimal planning}\label{sec:solution}

The robust planning problem described in \ref{eq:robust_planning} involves a max-min structure that makes it hard to determine the robustly optimal decision $b^*(\epsilon)$. In this section we follow an alternative path based on Lagrange duality theory~\cite{LP_over_distributions}. Under certain technical conditions the inner minimization in \ref{eq:robust_planning} over probability distributions $F$ of the unknown quantity $x$ is equivalent to its Lagrange dual (maximization) problem. Replacing then the inner minimization in \ref{eq:robust_planning} with the equivalent maximization yields a maximization problem (with a single optimization layer) over the planning decision variable $b$ and additional (dual) variables. We state this result in the following theorem. Its proof relies on the results of~\cite{LP_over_distributions}.%\footnote{The proofs of the results in this paper are omitted due to space limitations, but can be found in~\cite{Gatsis_proofs}.}

\begin{thm}\label{thm:dual}
If Assumption \ref{as:strict_inequality} holds, the robust planning problem defined in \ref{eq:robust_planning} is equivalent to the following optimization problem
\begin{alignat}{2}
	& \underset{ b, \lambda , \eta }{\text{maximize}} \quad && - \sum_{i=1}^n \lambda_i \, \epsilon_i - \eta  \label{eq:semi_infinite_program}\\
	& \text{subject to} \quad  
	&& J(x, b) + \sum_{i=1}^n \lambda_i \, g_i(x) +  \eta \geq 0, \quad \forall x \in \mathcal{X}, \label{eq:constraints_forall_x}\\
	& && b \in \mathcal{B} , \; \lambda \in \mathbb{R}_+^n, \; \eta \in \mathbb{R} \label{eq:last_line}.
\end{alignat}
\end{thm}

\begin{proof}
The proof relies on the strong duality results in~\cite{LP_over_distributions}. First consider, for any value of $b$, the inner minimization in \ref{eq:robust_planning}, in which the optimization variable is a probability distribution $F$ over the unit interval $\mathcal{X}$. Formally this optimization variable can be expressed as a signed measure $\mu$ on the space $\mathcal{X}$ equipped with the standard Borel $\sigma$-algebra. To be a probability measure, $\mu$ is required to be positive, denoted by $\mu \geq 0$, meaning that $\mu$ assigns a positive mass on any subset of $\mathcal{X}$ that belongs in the standard Borel $\sigma$-algebra, and also to have a total mass equal to $1$, $\int_x \, d\mu(x) = 1$. Under this change of variables from probability distributions $F$ to positive measures $\mu$, and by the form of the set $\mathcal{F}$ in \ref{eq:set_F_defn}, the inner minimization in \ref{eq:robust_planning} can be equivalently written as
\begin{alignat}{2}
	P(b) = \; & \underset{ \mu \geq 0 }{\text{minimize}} \quad && \int_x \, J(x,b) \, d\mu(x) \label{eq:measure_opt}\\
	& \text{subject to} \quad  
	&& \int_x \, g_i(x) \, d\mu(x) \leq \epsilon_i, \quad i=1, \ldots, n, \label{eq:conic_ineq}\\
	& \quad && \int_x \, d\mu(x) = 1. \label{eq:conic_eq}
\end{alignat}
The robust optimization problem \ref{eq:robust_planning} is equivalent to $P^*(\epsilon) = \sup_{b \in \mathcal{B}}\, P(b)$.

The problem \ref{eq:measure_opt}-\ref{eq:conic_eq} is a linear program over measures~\cite{LP_over_distributions}. It has $n+1$ constraints, stating that the $(n+1)$-dimensional vector $\int_x [g_1(x), \ldots, g_n(x), 1]^T d\mu(x)$ lies in a cone with a vertex at point $(\epsilon_1, \ldots, \epsilon_n, 1) \in \mathbb{R}^{n+1}$. We now claim that if we perturb this vertex point to be anywhere in an open ball around the point $(\epsilon_1, \ldots, \epsilon_n, 1)$, the resulting optimization problem of the form \ref{eq:measure_opt}-\ref{eq:conic_eq} is feasible.

\begin{claim}\label{prop:conic}
Let Assumption \ref{as:strict_inequality} hold. Then we can find an open ball $B \subseteq \mathbb{R}^{n+1}$ around the point $(\epsilon_1, \ldots, \epsilon_n, 1)$  such that for any $\epsilon' \in B$ the set of signed measures
\begin{alignat}{2}\label{eq:feasible_measures}
	\Big\{ \mu \geq 0: \, \begin{array}{l}
	\int_x \, g_i(x) \, d\mu(x) \leq \epsilon'_i, \, i=1, \ldots, n, \\
	\int_x \, d\mu(x) = \epsilon'_{n+1} 
	\end{array}
	\Big\}.
\end{alignat}
is non-empty.
\end{claim}

This proof of this claim is included in Appendix~\ref{sec:appendix_proof} and is a consequence of Assumption \ref{as:strict_inequality}.

We then derive the Lagrange dual problem of \ref{eq:measure_opt}-\ref{eq:conic_eq} by defining multipliers (dual variables) $\lambda_i \geq 0$ for each one of the inequality constraints \ref{eq:conic_ineq}, $i=1,\ldots,n$, as well as a multiplier $\eta$ for the equality constraint \ref{eq:conic_eq}. Since every finite subset of the space $\mathcal{X}$ belongs in the standard Borel $\sigma$-algebra, and all Dirac probability measures $\delta(x)$ at points $x \in \mathcal{X}$ are candidates for positive measures $\mu \geq 0$ in problem \ref{eq:measure_opt}, it can be shown -- see~\cite[p.11]{LP_over_distributions} -- that the Lagrange dual problem becomes
\begin{alignat}{2}
	D(b) = \;  &\underset{ \lambda \geq 0, \eta }{\text{maximize}} \quad && -\sum_{i=1}^n \lambda_i \epsilon_i - \eta  \label{eq:dual_problem}\\
	& \text{subject to} \quad  
	&& J(x, b) + \sum_{i=1}^n \lambda_i g_i(x) +  \eta \geq 0 , \notag\\
	&&&\quad \forall x \in \mathcal{X} .\label{eq:dual_constraint}
\end{alignat}

Then, \cite[Prop. 3.4]{LP_over_distributions} shows that under the result of Prop.~\ref{prop:conic} the optimal values of the minimization in \ref{eq:measure_opt} and the maximization in \ref{eq:dual_problem} are equal, i.e., $P(b) = D(b)$. This holds for any variable $b \in \mathcal{B}$. Hence, returning to the original problem \ref{eq:robust_planning} and recalling that it can be written as $P^*(\epsilon) = \sup_{b \in \mathcal{B}}\, P(b)$, we also have that $P^*(\epsilon) = \sup_{b \in \mathcal{B}}\, D(b)$. In other words we can replace the minimization over distributions $F$ with the maximization over dual variables $\lambda,\eta$ in \ref{eq:dual_problem}-\ref{eq:dual_constraint}. The resulting problem is a maximization over variables $b, \lambda, \eta$ and corresponds exactly to problem \ref{eq:semi_infinite_program}.
\end{proof}

According to the theorem, the optimal choice $b^*(\epsilon)$ of the robust planning problem \ref{eq:robust_planning} can be equivalently found by the optimization problem \ref{eq:semi_infinite_program}-\ref{eq:last_line}, and the optimal values of the two problems are equal. The advantage of the representation in \ref{eq:semi_infinite_program}-\ref{eq:last_line} is that it bypasses the max-min structure of \ref{eq:robust_planning}. There is a finite number of optimization variables in \ref{eq:semi_infinite_program}-\ref{eq:last_line}, i.e., the decision $b$, as well as the dual variables $\eta \in \mathbb{R}$ and the vector $\lambda \in \mathbb{R}^n_+$. The caveat, although, is that the number of constraints is infinite and uncountable. Such optimization problems are called semi-infinite -- see~\cite{semi_infinite_prog} for an overview. Nevertheless it is a convex optimization problem since the objective is linear and the constraint \ref{eq:constraints_forall_x} for each value of $x$ defines a convex set due to the concavity of function $J(x,b)$ in variable $b$ for any value $x$.

General approaches for solving semi-infinite programs can be found in~\cite{semi_infinite_prog}. We examine however next the special case in which the forecasts \ref{eq:forecast_inequality} are given in the form of prediction intervals, described in Example~\ref{example:prediction_intervals}. In that case it turns out that the number of constraints in the equivalent problem \ref{eq:semi_infinite_program}-\ref{eq:last_line} can be reduced to a finite number. Hence we can pose the robust planning decision under prediction interval forecasts as a standard finite-dimensional convex optimization problem, for which efficient algorithms exist.

After this special case, we show in the following section that the optimal values of the additional variables $\lambda$ in the optimization problem \ref{eq:semi_infinite_program}-\ref{eq:last_line} can be interpreted as indicators of how valuable are the given forecasts in determining the optimal planning. Based on this interpretation, we also examine in the following section how updates on the forecasts can yield a higher planning objective value.

\subsection{Robust planning under prediction intervals}\label{sec:planning_prediction_intervals}

Consider forecasts given in the form of prediction intervals described in \ref{eq:prediction_intervals}. By their reformulation into the generic form of forecasts presented in \ref{eq:standard_prediction_intervals}, we can pose the robust planning decision problem \ref{eq:robust_planning} into its equivalent form provided by Theorem~\ref{thm:dual} in \ref{eq:semi_infinite_program}-\ref{eq:last_line}. In particular, introducing dual variables $\bar{\lambda}_1, \ldots, \bar{\lambda}_m$ for the upper bound constraints in \ref{eq:standard_prediction_intervals}, and $\underline{\lambda}_{1},\ldots,\underline{\lambda}_{m}$ similarly for the lower bounds, we have that robust planning follows from the solution to the problem
\begin{alignat}{2}
	& \underset{b \in \mathcal{B} , \, \lambda \geq 0, \, \eta }{\text{maximize}} \quad 
	&& \sum_{i=1}^m \left( \underline{\lambda}_{i} \, \underline{\delta}_i - \bar{\lambda}_i \, \bar{\delta}_i  \right) - \eta  \label{eq:robust_planning_intervals}\\
	& \text{subject to} \quad  
	&&  J(x, b) + \sum_{i=1}^m (\bar{\lambda}_i - \underline{\lambda}_{i}) \, \mathds{1}\{[ x_{i-1}, x_{i}) \} \notag\\
	&&& \quad+  \eta \geq 0, \quad \text{for all} \; x \in [0,1] \label{eq:constraint}.
\end{alignat}

We note that the continuum of constraints \ref{eq:constraint} can be equivalently separated into the given intervals 
\begin{equation}\label{eq:constraint_2}
	 J(x, b) + (\bar{\lambda}_i - \underline{\lambda}_{i}) +  \eta \geq 0, \quad \text{for all} \; x \in [ x_{i-1}, x_{i}), 
\end{equation}
for all $i=1,\ldots,m$. Note also that by the assumption that the function $J(x,b)$ is concave in $x$ for all $b$, it follows that $J(x, b) \geq \min\{J(x_{i-1}, b), J(x_i, b)\}$, where the inequality is tight by continuity of function $J(x,b)$. Hence instead of checking \ref{eq:constraint_2} for a continuum of values $x \in [ x_{i-1}, x_{i})$ for each $i=1, \ldots, m$, we only need to check at the two endpoints $x_{i-1}$ and $x_i$. In other words, the problem \ref{eq:robust_planning_intervals}-\ref{eq:constraint} can be equivalently written as 
\begin{alignat}{2} \label{eq:planning_prediction_intervals}
	& \underset{ b \in \mathcal{B} , \, \lambda \geq 0, \, \eta }{\text{maximize}} \quad 
	&& \sum_{i=1}^m \left( \underline{\lambda}_{i} \, \underline{\delta}_i - \bar{\lambda}_i \, \bar{\delta}_i  \right) - \eta  \\
	& \text{subject to} \quad  
	&& J(x_{i-1}, b) + \bar{\lambda}_i - \underline{\lambda}_{i} +  \eta \geq 0 , \notag\\
	& && J(x_{i}, b) + \bar{\lambda}_i - \underline{\lambda}_{i} +  \eta \geq 0, \notag\\
	&&& \quad \quad \text{for all} \; i=1,\ldots,m. \label{eq:constraint_pred_int}
\end{alignat}
This optimization problem is convex and has a finite number of constraints. The number of constraints is twice the number of initial forecast intervals because essentially each forecast of the robust optimization is converted into two constraints, one for each of the two interval endpoints. The solution of \ref{eq:planning_prediction_intervals} can be determined efficiently by standard convex optimization algorithms.

\section{Relative value of forecasts}\label{sec:forecast_updates}

In this section we aim to quantify the value of the given forecasts \ref{eq:forecast_inequality} to the planner who needs to solve the robust planning problem \ref{eq:robust_planning}. In particular, \ref{eq:forecast_inequality} defines a set of $n$ given forecasts and we are interested in determining how valuable information does each one of them provide to the planner. To this end, we examine how sensitive the robustly optimal objective $\mathcal{P}^*(\epsilon)$, defined in \ref{eq:robust_planning}, is to the given forecast values $\epsilon$, i.e., what would the objective value become for deviations from the given parameters $\epsilon$. 

The forecast parameters $\epsilon$ determine the objective $P^*(\epsilon)$ defined in \ref{eq:robust_planning} via the set of distributions $\mathcal{F}(\epsilon)$ appearing in the constraint of the inner minimization. To examine the sensitivity of $P^*(\epsilon)$ when the given forecast parameters $\epsilon$ change to some new values $\epsilon - \Delta \epsilon$ for some $\Delta \epsilon \in \mathbb{R}^n$, with the negative sign chosen for convention, we leverage the problem reformulation provided by Theorem \ref{thm:dual}. We exploit the well-known convex optimization fact that in general the optimal values of the Lagrange dual variables $\lambda$ express the sensitivity of the objective of a convex optimization problem with respect to changes in the constraints - see~\cite[Ch. 5.6]{Boyd_convex}. In our case, the robust planning \ref{eq:robust_planning} involves a two-layer (max-min) optimization objective, but a similar sensitivity result can be obtained. We state this result in the following theorem.%, and we include the proof for completeness.

%The dual variables $\lambda$ can be thought as local derivatives of the objective value $P^*(\mathcal{F} (\epsilon)) = \, \inf_{F \in \mathcal{F}(\epsilon)} \; \mathcal{J}(F,b^*)$ with respect to the forecast parameters $\epsilon$ at the chosen optimal point $b^*$. Their magnitude can be interpreted as an indicator on how to update the forecast (increase or decrease the parameters) in a way that increases the value $P^*(\mathcal{F} (\epsilon))$.

\begin{thm}\label{thm:sensitivity}
Under Assumption \ref{as:strict_inequality}, for any $\Delta \epsilon \in \mathbb{R}^n$ it holds that
\begin{equation}\label{eq:utility_improve}
	P^*(\epsilon - \Delta \epsilon) \geq P^*(\epsilon) + \sum_{i=1}^n \lambda_i^*(\epsilon) \, \Delta \epsilon_i ,
\end{equation}
where $\lambda_i^*(\epsilon)$, for $i=1,\ldots,n$, is the optimal solution to problem \ref{eq:semi_infinite_program}.
\end{thm}

\begin{proof}
Along with $\lambda_i^*(\epsilon)$ and already defined $b^*(\epsilon)$, let $\eta^*(\epsilon)$ be a corresponding optimal solution to problem \ref{eq:semi_infinite_program}-\ref{eq:last_line}. Under Assumption \ref{as:strict_inequality} by Theorem \ref{thm:dual}, i.e., the equivalence of \ref{eq:robust_planning} and \ref{eq:semi_infinite_program}-\ref{eq:last_line}, we have for the optimal objective that 
\begin{equation}\label{eq:dual_optimality}
	P^*(\epsilon) = - \sum_{i=1}^n \lambda_i^*(\epsilon) \, \epsilon_i - \eta^*(\epsilon) ,
\end{equation}
as well as
\begin{equation} \label{eq:dual_feasibility}
	J(x, b^*(\epsilon)) + \sum_{i=1}^n \lambda_i^*(\epsilon) \, g_i(x) +  \eta^*(\epsilon) \geq 0,
\end{equation}
for all $x \in \mathcal{X}$ by the feasibility to constraints \ref{eq:constraints_forall_x}.

Fix any $\Delta \epsilon \in \mathbb{R}^n$ and consider $P^*(\epsilon - \Delta \epsilon)$ defined by \ref{eq:robust_planning}. By definition we have that
\begin{alignat}{1}
	P^*(\epsilon - \Delta \epsilon)
	&= \sup_{b \in \mathcal{B}}\; \inf_{F \in \mathcal{F}(\epsilon - \Delta \epsilon)} \; \mathbb{E}_{x \sim F} \, J(x,b) \notag\\
	&\geq \inf_{F \in \mathcal{F}(\epsilon - \Delta \epsilon)} \; \mathbb{E}_{x \sim F} \, J(x,b^*(\epsilon)) ,
	\label{eq:rhs}
\end{alignat}
where the last inequality follows because $b^*(\epsilon)$ is in general a suboptimal choice for $b \in \mathcal{B}$. Following arguments similar to those in the proof of Theorem \ref{thm:dual} we can show that the optimal value of the minimization on the right hand side of \ref{eq:rhs} is lower bounded by the optimal value of its Lagrange dual maximization problem\footnote{In fact in Theorem \ref{thm:dual} we show that under Assumption \ref{as:strict_inequality} the two optimal values are equal for $\Delta \epsilon = 0$, but here for $\Delta \epsilon \neq 0$ the equality might not hold.}. Replacing this lower bound in \ref{eq:rhs} we have that 
\begin{alignat}{2}
	P^*(\epsilon - \Delta \epsilon) 
	\geq \; &\underset{ \lambda \geq 0 , \eta }{\text{maximize}} \quad && - \sum_{i=1}^n \lambda_i \, (\epsilon_i - \Delta \epsilon_i) - \eta  \label{eq:semi_infinite_program_2}\\
	& \text{subject to} \quad  
	&& J(x, b^*(\epsilon)) + \sum_{i=1}^n \lambda_i \, g_i(x) \notag\\
	&&& \quad +  \eta \geq 0, \quad \text{for all} \; x \in \mathcal{X}. \label{eq:constraints_forall_x_2}
\end{alignat}
Note now that $\lambda^*(\epsilon)$ and $\eta^*(\epsilon)$ are feasible solutions for the last optimization problem as follows by \ref{eq:dual_feasibility}, hence we have that
\begin{alignat}{1}
	P^*(\epsilon - \Delta \epsilon)
	&\geq  - \sum_{i=1}^n \lambda_i^*(\epsilon) \, (\epsilon_i - \Delta \epsilon_i) - \eta^*(\epsilon)  \notag\\
	&= P^*(\mathcal{F}(\epsilon)) + \sum_{i=1}^n \lambda_i^*(\epsilon) \, \Delta \epsilon_i
\end{alignat}
where the last equality follows by \ref{eq:dual_optimality}. Consequently the statement of the theorem holds.
\end{proof}

The theorem provides a lower bound on the optimal planning objective value of \ref{eq:robust_planning} after changes in the forecast bounds $\epsilon$ provided in \ref{eq:forecast_inequality}. The change in the lower bound compared to $P^*(\epsilon)$ is proportional to the change from $\epsilon$ to $\epsilon - \Delta \epsilon$. The rate of change in planning objective units in each direction $i$ is given by the optimal value of the (dual) variable $\lambda_i^*(\epsilon)$ for each element $i=1,\ldots,n$. We can thus interpret the values $\lambda_i^*(\epsilon)$ as indicators of how \textit{valuable} each one of the forecasts \ref{eq:forecast_inequality} is to the planner. A forecast $i$ with a high value $\lambda_i^*(\epsilon)$ is more informative compared to other forecasts $j \neq i$ because a small change in the given forecast bound $\epsilon_i$ gives a significant change in the (lower bound of the) objective. We also note that the values $\lambda_i^*(\epsilon)$ have already been obtained during the computation of the robust planning by \ref{eq:semi_infinite_program}, hence no further computation is required to determine them.

\begin{remark}
The sensitivity analysis provided by Theorem~\ref{thm:sensitivity} is performed with respect to the given forecast parameters $\epsilon$, i.e., they are obtained locally for $\epsilon$. At some other forecast parameter $\epsilon'$, the sensitivities, i.e., the values $\lambda_i^*(\epsilon')$ will differ. This difference matches the intuition that different forecasts provide different information. On the other hand the theorem provides a lower bound on the objective value at any deviation $\epsilon - \Delta \epsilon \in \mathbb{R}^n$, not just locally around $\epsilon$. It is also worth noting that the theorem does not provide any guarantee on how far the lower bound is with respect to the actual objective for such deviations. 
\end{remark}

Given the sensitivity-based characterization of how informative some given forecasts are, we next propose a method for refining/updating the forecasts in a way that increases the utility of the decision-making problem.

\subsection{Forecast refinements}\label{sec:forecast_refinement}

Suppose the decision-maker has access to a forecasting oracle, which upon request can refine one of the $n$ given forecasts in \ref{eq:forecast_inequality} -- see Remark~\ref{rem:oracle}. Inspecting the particular type of forecasts we consider in \ref{eq:forecast_inequality} written as inequalities, a refined forecast of some inequality $i$ can be represented by decreasing the bound $\epsilon_i$ to a lower value $\epsilon_i - \Delta \epsilon_i$. Equivalently, such decrease has the effect of shrinking the set $\mathcal{F}(\epsilon)$ of possible probability distributions that the decision-maker has to consider, as defined in \ref{eq:set_F_defn}, to a subset $\mathcal{F}(\epsilon - \Delta \epsilon) \subseteq \mathcal{F}(\epsilon)$. 

A question that arises in this context is how the decision-maker should select which of the $n$ forecasts of \ref{eq:robust_planning} to refine. This question is particularly important if the cost of refining forecasts is high, e.g., the computational cost of running extra simulations of a forecasting algorithm. The characterization of Theorem~\ref{thm:sensitivity} suggests that the forecast that provides the most valuable information should be selected, i.e., the forecast $i$ with the highest value $\lambda_i^*(\epsilon)$. Based on this intuition we propose the iterative forecast refinement procedure shown in Algorithm~\ref{alg:update}. On every iteration sensitivity values are obtained by solving the planning problem, and a refinement for the forecast bound with the highest sensitivity is requested (e.g. from the oracle) while all other bounds are kept constant. The procedure terminates, for example, if no further refinements are possible, or if the objective value improvements become insignificant.

%\begin{enumerate}
%%
%\item At iteration $k$, based on current forecast parameters $\epsilon(k) \in \mathbb{R}^n$, solve the robust planning problem in \ref{eq:semi_infinite_program}-\ref{eq:last_line} to determine the optimal planning $b^*(\epsilon(k))$ and the optimal dual variables $\lambda^*(\epsilon(k))$.
%%
%\item Select $j = \text{argmax}_{1\leq i \leq n} \; \lambda_i^*(\epsilon(k))$
%%
%\item Refine $\epsilon_j(k+1) \leq \epsilon_j(k)$, while keeping $\epsilon_i(k+1) = \epsilon_i(k)$ for all $i \neq j$
%%
%\end{enumerate}

\begin{algorithm}[!t]
\caption{Sensitivity-driven forecast refinement}\label{alg:update}
\begin{algorithmic}[1]

\Require Utility function $J(x,b)$, forecast functions $g_i(x)$ for $i=1,\ldots,n$, initial forecast parameters $\epsilon(0) \in \mathbb{R}^n$
\Ensure Robustly optimal planning decision $b$
%\Comment{...}
%\Statex
%\State Obtain initial forecast parameters $\epsilon(0) \in \mathbb{R}^n$
\State $k \gets 0$

%\Loop{iteration k}
\Repeat

\State Solve robust planning \ref{eq:semi_infinite_program}-\ref{eq:last_line} with forecast parameters $\epsilon(k)$, and determine optimal planning $b^*(\epsilon(k))$ and sensitivity values $\lambda^*(\epsilon(k)) \in \mathbb{R}^n$

\State Select $j = \text{argmax}_{1\leq i \leq n} \; \lambda_i^*(\epsilon(k))$

\State Request refinement in $j$th forecast $\epsilon_j(k+1) \leq \epsilon_j(k)$

\State Keep $\epsilon_i(k+1) = \epsilon_i(k)$ for all $i \neq j$

\State $k \gets k+1$
%\EndLoop
\Until{Termination condition}

\State \Return Final planning decision $b^*(\epsilon(k))$

\end{algorithmic}
\end{algorithm}

Even though this procedure is motivated mainly by intuition, without theoretical claims in terms of optimality, convergence, etc., we note the following fact. On each iteration, according to \ref{eq:utility_improve} of Theorem~\ref{thm:sensitivity}, the robustly optimal utility value increases by at least an amount equal to $\lambda_j^*(\epsilon(k)) \, (\epsilon_j(k+1) - \epsilon_j(k))$. This minimum increase is proportional to the amount of change in the selected forecast parameter. If the selected forecast $j$ cannot be refined, i.e., $\epsilon_j(k+1) = \epsilon_j(k)$, then no improvement in decision-making is achieved. In general, however, the result matches the intuition that with a better forecast, i.e., imposing a more constraining set $\mathcal{F}(\epsilon - \Delta \epsilon)$ in the minimization of \ref{eq:robust_planning}, the objective value of the planning can only improve, i.e., the optimal value of \ref{eq:robust_planning} cannot decrease.

In the following section we present a numerical example for a wind power plant participating in an electricity market (Example~\ref{example:bidding}) under prediction interval forecasts (Example~\ref{example:prediction_intervals}). We demonstrate the methodology derived by Theorem~\ref{thm:dual} for determining the robustly optimal bidding decision. Additionally we perform the sensitivity analysis developed in this section and we implement the sensitivity-driven procedure for refining forecasts. 

\begin{remark}\label{rem:oracle}
The methodology for refining forecasts in this section is developed without considering a specific forecasting procedure. If, for example, an oracle obtains the forecasts by sampling from the true distribution of the unknown quantity, through simulations of an underlying stochastic model, then forecasts can be refined by drawing extra samples. However, a large number of extra samples might be required to improve upon the estimates of all the integrals in \ref{eq:forecast_inequality}, e.g., in a Monte Carlo fashion. By focusing on only one of the $n$ integrals during refinement, as the oracle does in our hypothesis, the number of required samples can be reduced. This can be performed by variance reduction techniques, such as importance sampling~\cite{Ripley_simulation}, which aim at sampling more often from values important in estimating the selected integral.
\end{remark}

\section{Application: Bidding in electricity markets under prediction intervals}\label{sec:numerical}

\begin{figure}[!t]
  \centering
  \includegraphics[width=1.0\columnwidth]{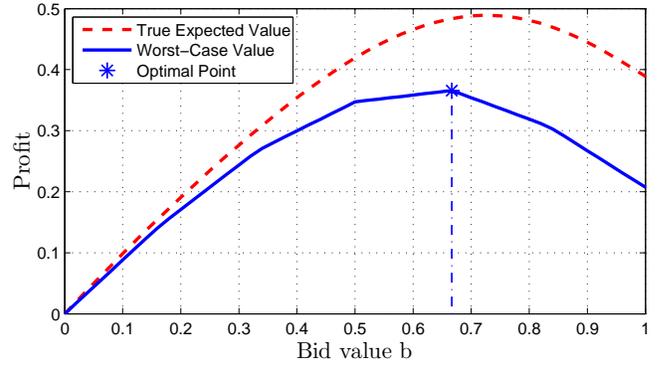}
  \caption{Expected profit and worst-case expected profit as a function of the chosen bidding for the case considered in the numerical example. The worst-case value is always an under-approximation of the true one. The robustly optimal bidding balances between the shortfall penalties and gains from a high bid.}
  \label{fig:profit}
\end{figure}

In this section we apply our theoretical results in a numerical example of bidding renewable energy in electricity markets. This problem was introduced in Example~\ref{example:bidding}, with the utility function representing the profit for the renewable generator and repeated here for convenience,
\begin{equation}\label{eq:J_bidding_repeat}
	J(x,b) = p\,b - q [b-x]_+ .
\end{equation}
The utility involves a reward term proportional to the bid $b$ with a rate $p>0$, and a penalty proportional to the generation shortfall $[b-x]_+$ with a rate $q > p$. Here both $x \in [0,1]$ and $b \in [0,1]$. Suppose also the generator obtains forecasts about the future generation value $x$ in the form of prediction intervals, introduced in Example~\ref{example:prediction_intervals}. We consider the problem of determining the value of bidding that robustly maximizes the expected profit to the generator subject to the given forecasts about the renewable generation. This problem is mathematically described by the robust planning formulation \ref{eq:robust_planning}. To solve this problem we adopt the reformulation into a convex optimization problem which was presented in Section \ref{sec:planning_prediction_intervals}, and in particular takes the form \ref{eq:planning_prediction_intervals}-\ref{eq:constraint_pred_int}.

In our numerical example, we consider forecasts for the probability that the random generation $x$ takes values in each of $m$ equal intervals that partition the total range of values $[0,1]$, i.e., $x_i = i/m$ for $i=0,\ldots,m$ -- see also Example~\ref{example:prediction_intervals}. To derive the lower and upper bounds, $\underline{\delta}_i$ and $\bar{\delta}_i$ respectively according to \ref{eq:prediction_intervals}, we adopt the following procedure. We fix some true underlying probability distribution of the generation $x$ and we produce bounds $\underline{\delta}_i, \bar{\delta}_i$ by perturbing the true value $\mathbb{P}(x_{i-1} \leq x < x_{i} )$ computed with the underlying distribution. Specifically we consider perturbation by a constant amount below and above $\mathbb{P}(x_{i-1} \leq x < x_{i} )$ for each $i=1,\ldots, m$. Then only the resulting lower and upper bounds, $\underline{\delta}_i$ and $\bar{\delta}_i$ respectively, are available to the generator, not the underlying probability distribution used to create them.

%\begin{figure}[!t]
  %\centering
  %\includegraphics[width=1.0\columnwidth]{prediction_interval_numerical.eps}
  %\caption{Prediction intervals and the underlying probability density used to generate them in the numerical example. The true probability that generation $x$ takes value in each of the $m=6$ equal intervals is perturbed by a constant amount to provide upper and lower bounds as in \ref{eq:prediction_intervals}.}
  %\label{fig:prediction_interval_numerical}
%\end{figure}

Given $m=6$ prediction intervals, shown in Fig.~\ref{fig:prediction_intervals}, we first solve the robust bidding problem, i.e., the convex optimization \ref{eq:planning_prediction_intervals}-\ref{eq:constraint_pred_int}, for reward and penalty values $p=1, q = 1.6$ in the profit function \ref{eq:J_bidding_repeat}. We obtain an optimal bidding value $b^*$ roughly equal to $0.67$. In Fig.~\ref{fig:profit} we illustrate the value of the worst-case expected profit, i.e., the objective in the planning problem \ref{eq:robust_planning}, for all bid values $b$ as well as the optimal point. For comparison we also illustrate the expected profit computed using the underlying true probability distribution of $x$ for all bid values $b$. The worst-case expected profit is an under-approximation of the true expected value, since the forecasts provide incomplete information about the underlying distribution. The optimal choice $b^*$ balances between the two extremes $0$ and $1$, as expected by intuition, with $b = 0$ being the most ''secure'' bid but yielding a zero expected profit, and $b=1$ being the most ''risky'' bid since it is certain that a shortfall $x<b=1$ will occur.

\begin{table}[!t]
\centering                                                                                                                 
\begin{tabular}{|c|c|c|c|c|c|c|}                                                                                           
\hline 
&&&&&&\\[-5pt]
& $\bar{\lambda}_1$ & $\bar{\lambda}_2$ & $\bar{\lambda}_3$ 
& $\bar{\lambda}_4$ & $\bar{\lambda}_5$ & $\bar{\lambda}_6$ \\
[5pt]
\hline                                                                                                                     
Iter 1 & 0.66 & 0.39 & 0.12 & 0.00 & 0.00 & 0.00 \\                                                                        
\hline                                                                                                                     
Iter 2 & 0.80 & 0.53 & 0.27 & 0.00 & 0.00 & 0.00 \\                                                                        
\hline                                                                                                                     
Iter 3 & 0.80 & 0.53 & 0.27 & 0.00 & 0.00 & 0.00 \\                                                                        
\hline                                                                                                                     
Iter 4 & 1.07 & 0.80 & 0.53 & 0.27 & 0.00 & 0.00 \\                                                                        
\hline                                                                                                                     
\end{tabular}                                                                                                              
\caption{Sensitivities to upper forecast bounds}
\label{tab:lambda_upper}                                                                                                 
\end{table}

\begin{table}[!t]
\centering                                                                                                                 
\begin{tabular}{|c|c|c|c|c|c|c|}                                                                                           
\hline 
&&&&&&\\[-5pt]   
 & $\underline{\lambda}_1$ & $\underline{\lambda}_2$ & $\underline{\lambda}_3$ & $\underline{\lambda}_4$ & $\underline{\lambda}_5$ & $\underline{\lambda}_6$ \\
[5pt]
\hline                
    Iter 1 & 0.00 & 0.00 & 0.00 & 0.14 & 0.41 & 0.41 \\
    \hline
    Iter 2 & 0.00 & 0.00 & 0.00 & 0.00 & 0.27 & 0.27 \\
    \hline
    Iter 3 & 0.00 & 0.00 & 0.00 & 0.00 & 0.27 & 0.27 \\
    \hline
    Iter 4 & 0.00 & 0.00 & 0.00 & 0.00 & 0.00 & 0.00 \\
\hline
\end{tabular}
\caption{Sensitivities to lower forecast bounds}
\label{tab:lambda_lower}
\end{table}

Moreover, we examine how the information provided by the given forecasts of Fig.~\ref{fig:prediction_intervals} is valued by the decision maker. Following the development of Section~\ref{sec:forecast_updates}, we examine the sensitivity of the robust bidding profit to the set of given forecasts. According to Theorem~\ref{thm:sensitivity} this sensitivity is captured by the optimal values of the dual variables. In particular the sensitivity to upper and lower bounds of prediction intervals, $\bar{\delta}_i$ and $\underline{\delta}_i$ respectively, for $i=1,\ldots, m$, is captured by the dual variables $\bar{\lambda}_1, \ldots, \bar{\lambda}_m$ and $\underline{\lambda}_{1}, \ldots, \underline{\lambda}_{m}$ respectively, as introduced in Section~\ref{sec:planning_prediction_intervals}. 

The sensitivity values in our example, as computed by solving \ref{eq:planning_prediction_intervals}-\ref{eq:constraint_pred_int}, are shown in the first row of each of the Tables~\ref{tab:lambda_upper},~\ref{tab:lambda_lower}. First we observe that at each interval $i=1,\ldots,m$, at most one of the upper or lower bound sensitivities $\bar{\lambda}_i, \underline{\lambda}_i$ is non-zero. The reason is that at most one of two bounds (cf.~\ref{eq:prediction_intervals}) is tight for the worst-case distribution at the robustly optimal point of problem \ref{eq:robust_planning}. Second, from the dual values we see that the optimal profit is sensitive to the probability of both having a low generation (captured by $\bar{\lambda}_1, \bar{\lambda}_2,  \bar{\lambda}_3$) as well as having a high generation (captured by $\underline{\lambda}_4, \underline{\lambda}_5, \underline{\lambda}_6$). In other words information about these events is the most informative. However the sensitivity is higher in the low generation case because the profit is significantly affected by shortfalls in this case.

%Figure~\ref{fig:lambda} shows the values of all the sensitivity values, i.e., the dual variables, for the two scenarios. First we observe that at each interval $i=1,\ldots,m$, only one of the upper or lower bound of the forecast (cf.~\ref{eq:prediction_intervals}) is non-zero. The reason is that only one of two bounds is tight for the worst-case distribution at the optimal point in the robust planning problem \ref{eq:robust_planning}. In the first scenario, where the shortfall penalty is low, the optimal profit is sensitive to the probability of both having a low generation as well as having a high generation. In other words, information about these events is the most informative. In the second scenario on the other hand, where the cost of shortfall is high, the most informative forecasts are the ones corresponding to a low power generation because on that event the profit is significantly affected. Overall, these observations indicate that the same forecast can be of different value to planning problems with different objectives, also leading to different planning decisions. 

Next we adopt the sensitivity-driven methodology developed in Section~\ref{sec:forecast_updates} that uses a forecasting oracle to refine the given forecasts so that they become more informative for the robust bidding problem. Following Algorithm~\ref{alg:update} we iteratively select to refine the forecast with the highest sensitivity value, i.e., one of the upper or lower bounds $\bar{\delta}_i$ and $\underline{\delta}_i$ for an interval $i$ in \ref{eq:prediction_intervals}. If an upper bound $\bar{\delta}_i$ is selected it gets decreased, and respectively if a lower bound $\underline{\delta}_i$ is selected it gets increased. In our numerical example, the resulting new bound needs to be consistent with the true underlying distribution, so we model the following forecasting oracle. Upon request, the oracle can refine each upper or lower bound by a constant decrease or increase respectively, and the new bound is close enough to the true value so that the oracle cannot refine it any further. This model is chosen here for simplicity. In general, Algorithm~\ref{alg:update} can be applied regardless of how the forecasting oracle operates.

\begin{figure}[!t]
  \centering
  \includegraphics[width=1.0\columnwidth]{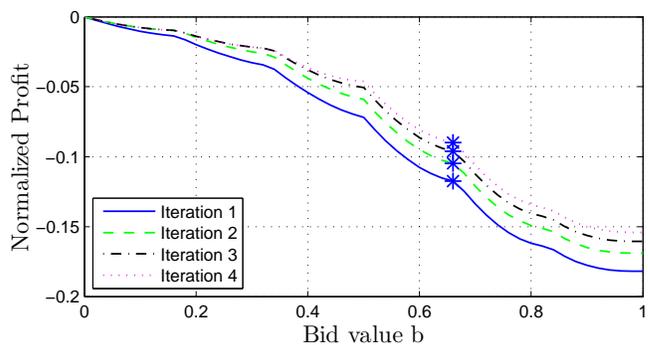}
  \caption{Normalized worst-case profit as a function of bid value after iterations of the forecast refinement algorithm. The plotted values are normalized by subtracting the true underlying expected profit. The refinements increase the worst-case profit and bring it closer to the true value. The robustly optimal biding values are also indicated.}
  \label{fig:update}
\end{figure}

We perform iterations of Algorithm~\ref{alg:update} and the corresponding sensitivity values on each iteration are shown in Tables~\ref{tab:lambda_upper},~\ref{tab:lambda_lower}. As already mentioned, $\bar{\lambda}_1$ has the higher value initially, so the upper bound on the first interval is refined (i.e., reduced). On the second iteration, after solving again the optimal planning \ref{eq:planning_prediction_intervals}-\ref{eq:constraint_pred_int}, the new sensitivity values are obtained and again $\bar{\lambda}_1$ is the largest. Since the forecasting oracle described in our example cannot reduce the corresponding bound any further, the bound with the second largest sensitivity is selected to be refined, which is $\bar{\lambda}_2$. On the third iteration, again $\bar{\lambda}_1, \bar{\lambda}_1$ are the largest, but due to the oracle model, the third largest $\bar{\lambda}_3$ is selected, and so on. It is worth noting that after some iterations we observe that the sensitivities of all lower bounds $\underline{\lambda}_i$ become zero. This means that currently they are not providing any valuable information, and refining (i.e., increasing) them does not necessarily offer an improvement in profit.

Finally, to examine the improvements in the bidding profit after these iterations of the refinement algorithm, we plot in Fig.~\ref{fig:update} the worst-case expected profit, i.e., the objective in the planning problem \ref{eq:robust_planning}, for all bid values $b$. For visualization reasons the values are normalized by subtracting the true expected profit (the latter is larger, as shown in Fig.~\ref{fig:profit}, so the plotted values are negative). We observe that as the forecasts are refined, the worst-case expected profit increases and gets closer to the actual expected profit (the baseline zero). In other words, the refined forecasts provide a more accurate model of the true profit. It is worth noting that the largest discrepancy between worst-case and true expected value happens at the high bid region, since the worst-case scenario assumes that generation takes the lowest possible values with the highest possible probability, making the associated shortfall costs large. We also note that in this example even though the optimal value of the profit increases with the forecast refinements, the optimal bidding value shown in Fig.~\ref{fig:update} is the same. This is a consequence of the piecewise linear utility function in \ref{eq:J_bidding_repeat}. For more general utility functions the optimal planning decision can change as well during the refinements, along with the optimal objective values, and our algorithm is able to track these changes.

%We consider the second scenario, with higher shortfall penalty $q = 2$. After each of three iterations of the forecast refinement Algorithm~\ref{alg:update} we plot in Fig.~\ref{fig:update} the worst-case expected profit, i.e., the objective in the planning problem \ref{eq:robust_planning}, for all bid values $b$ as well as the respective optimal points. We observe that as the forecasts are refined, the worst-case expected profit increases and gets closer to the actual expected profit. In other words, the refined forecasts provide a more accurate model of the profit function. We note that in this example even though the optimal value of the profit increases with the forecast refinements, the optimal bidding value remains the same. The reason behind such invariance is that the profit function \ref{eq:J_bidding_repeat} in this example is piecewise linear. The resulting robust planning problem \ref{eq:planning_prediction_intervals} is a linear program and the optimal bidding happens to correspond to the same vertex point during the refinements. In general, the optimal planning decision could change during the refinements, along with the optimal objective values.

\section{Concluding remarks and future work}\label{sec:remarks}

We examine how planning problems under uncertainty can utilize forecasts about uncertain quantities. By equivalently expressing forecasts as constraints on the possible probability distributions that the uncertain quantity can follow, the problem is cast as a robust optimization. The optimal planning seeks to maximize the expected value of a given utility function, integrated with respect to the worst-case distribution consistent with the forecasts. A reformulation into a convex optimization problem is presented, from which we can extract information about how valuable are the forecasts in determining the optimal planning decision.  

Under the hypothesis that a forecasting oracle can refine the given forecasts upon request, i.e., that the set of possible probability distributions in the robust planning can be decreased, an iterative forecast refinement procedure is proposed in Section~\ref{sec:forecast_refinement}. Even though the procedure is justified theoretically, in practice the way refinements are performed depends on the employed forecasting algorithm. For example, it might be the case that many forecasts can be refined at the same time, instead of just one at a time as we consider, or that the requested forecast cannot be refined due to limitations of the oracle and/or because the forecast is very close to the true value. Moreover, besides just refining the set of given forecasts, as we consider, the oracle might generate new additional forecasts. Overall, coupling the proposed sensitivity-driven approach with forecasting algorithms used in practice requires further exploration.

More general problem formulations include the study of uncertain quantities evolving over time in a potentially correlated fashion, e.g., the random renewable generation during a time horizon in the future. Forecasting for such models has been considered in, e.g., scenario-based forecasts~\cite{Pinson_scenario}. Furthermore, apart from uncertain quantities spread over time, the framework can be extended to include spatially distributed models as well, as in the case of correlated renewable generations at different physical locations in the grid.

\appendix
\subsection{Proof of Claim 1}\label{sec:appendix_proof}

The proposition equivalently states that we can find $\delta >0$ (the radius of ball $B$) such that for any $\epsilon' \in \mathbb{R}^{n+1}$ that satisfies $\vert \epsilon_i' - \epsilon_i \vert < \delta$, $i=1,\ldots,n$, and $ \vert \epsilon_{n+1}' - 1 \vert < \delta$, i.e., $\epsilon'$ belongs in ball $B$, the set \ref{eq:feasible_measures} is non-empty. The proof is constructive. We first construct an appropriate ball radius $\delta$, and then for any point $\epsilon' \in B$ in the ball we construct a measure $\mu$ in the set \ref{eq:feasible_measures}, i.e., satisfying
\begin{alignat}{2}
	&\int_x \, g_i(x) \, d\mu(x) \leq \epsilon_i' , \quad i=1, \ldots, n, \label{eq:perturbed_ineq}\\
	&\int_x \, d\mu(x) = \epsilon_{n+1}'. \label{eq:perturbed_eq}
\end{alignat}

First, by Assumption \ref{as:strict_inequality} we have that there exists a probability measure $\bar{\mu} \geq 0$ such that 
\begin{alignat}{2}
	&\int_x \, g_i(x) \, d\bar{\mu}(x) \leq \epsilon_i - \zeta , \quad i=1, \ldots, n, \label{eq:strict_inequality}\\
	&\int_x \, d\bar{\mu}(x) = 1,
\end{alignat}
for some $\zeta>0$. Define the ball radius to be
\begin{equation}\label{eq:delta_choice}
	\delta = \min\; \left\{ \frac{\zeta}{1 + | \epsilon_1 - \zeta|},\, \ldots,  
	\, \frac{\zeta}{1 + | \epsilon_n - \zeta|}, \, 1 \right\}. 
\end{equation}
%
%%
%\begin{equation}\label{eq:delta_choice}
%	\delta = \min_{i=1,\ldots,n} \; \frac{\zeta}{1 + | \epsilon_i - \zeta|} . 
%\end{equation}
%%

Fix then any point $\epsilon' \in B$. We need to show that there exists a signed measure $\mu \geq 0$ such that \ref{eq:perturbed_ineq} and \ref{eq:perturbed_eq} hold. Consider the measure $\mu = \epsilon_{n+1}' \, \bar{\mu}$ defined by scaling the measure $\bar{\mu}$, i.e., $\mu(S) = \epsilon_{n+1}'\, \bar{\mu}(S)$ at each set $S$ in the Borel $\sigma$-algebra of the set $\mathcal{X}$. Note that this measure is positive since $\epsilon_{n+1}' > 0$, which follows from $ \vert \epsilon_{n+1}' - 1 \vert < \delta \leq 1$. Also this measure satisfies \ref{eq:perturbed_eq} since $\int_x \, d\mu(x) = \epsilon_{n+1}'\, \int_x d\bar{\mu}(x) = \epsilon_{n+1}'$.

Hence it remains to show that \ref{eq:perturbed_ineq} holds. The right hand side of \ref{eq:perturbed_ineq} satisfies $\epsilon_i' \geq \epsilon_i - \delta$. Also, the left hand side of \ref{eq:perturbed_ineq} satisfies
\begin{equation}\label{eq:temp}
	\int_x \, g_i(x) \, d{\mu}(x) = \epsilon_{n+1}'\, \int_x \, g_i(x) \, d\bar{\mu}(x) \leq \epsilon_{n+1}'\, (\epsilon_i - \zeta),
\end{equation}
because of \ref{eq:strict_inequality} and $\epsilon_{n+1}' >0$ as we argued previously. Hence to show \ref{eq:perturbed_ineq} it suffices to show that $\epsilon_{n+1}'\, (\epsilon_i - \zeta) \leq \epsilon_i - \delta$ for all $i=1,\ldots,n$.

We consider two cases for the sign of quantity $\epsilon_i - \zeta$. If $\epsilon_i - \zeta \geq 0$, and since $\epsilon_{n+1}' < 1 + \delta$, we have that $\epsilon_{n+1}' (\epsilon_i - \zeta) \leq (1 + \delta) (\epsilon_i - \zeta)$. In that case to show \ref{eq:perturbed_ineq} it suffices to show that $(1 + \delta) (\epsilon_i - \zeta) \leq \epsilon_i - \delta$. This is equivalent to $\delta ( 1+ \epsilon_i - \zeta) \leq \zeta$, easily checked by our choice of $\delta$ in \ref{eq:delta_choice}.

Similarly for the case $\epsilon_i - \zeta < 0$, since $\epsilon_{n+1}' > 1 - \delta$, we have that $\epsilon_{n+1}' (\epsilon_i - \zeta) \leq (1 - \delta) (\epsilon_i - \zeta)$. Thus for \ref{eq:perturbed_ineq} it suffices to show that $(1 - \delta) (\epsilon_i - \zeta) \leq \epsilon_i - \delta$ which can be again easily checked by our choice of $\delta$ in \ref{eq:delta_choice}.

\bibliography{biblio_energy}

\begin{thebibliography}{10}
\providecommand{\url}[1]{#1}
\csname url@samestyle\endcsname
\providecommand{\newblock}{\relax}
\providecommand{\bibinfo}[2]{#2}
\providecommand{\BIBentrySTDinterwordspacing}{\spaceskip=0pt\relax}
\providecommand{\BIBentryALTinterwordstretchfactor}{4}
\providecommand{\BIBentryALTinterwordspacing}{\spaceskip=\fontdimen2\font plus
\BIBentryALTinterwordstretchfactor\fontdimen3\font minus
  \fontdimen4\font\relax}
\providecommand{\BIBforeignlanguage}[2]{{%
\expandafter\ifx\csname l@#1\endcsname\relax
\typeout{** WARNING: IEEEtran.bst: No hyphenation pattern has been}%
\typeout{** loaded for the language `#1'. Using the pattern for}%
\typeout{** the default language instead.}%
\else
\language=\csname l@#1\endcsname
\fi
#2}}
\providecommand{\BIBdecl}{\relax}
\BIBdecl

\bibitem{mills2011}
A.~Mills, M.~Ahlstrom \emph{et~al.}, ``Dark shadows,'' \emph{IEEE Power and
  Energy Magazine}, vol.~9, no.~3, pp. 33--41, 2011.

\bibitem{makarov2009}
Y.~V. Makarov, C.~Loutan, J.~Ma, and P.~de~Mello, ``Operational impacts of wind
  generation on {C}alifornia power systems,'' \emph{IEEE Transactions on Power
  Systems}, vol.~24, no.~2, pp. 1039--1050, 2009.

\bibitem{risk_limiting}
P.~P. Varaiya, F.~F. Wu, and J.~W. Bialek, ``Smart operation of smart grid:
  Risk-limiting dispatch,'' \emph{Proceedings of the IEEE}, vol.~99, no.~1, pp.
  40--57, 2011.

\bibitem{stochastic_UC}
E.~M. Constantinescu, V.~M. Zavala, M.~Rocklin, S.~Lee, and M.~Anitescu, ``A
  computational framework for uncertainty quantification and stochastic
  optimization in unit commitment with wind power generation,'' \emph{IEEE
  Transactions on Power Systems}, vol.~26, no.~1, pp. 431--441, 2011.

\bibitem{BitarEtal}
E.~Bitar, R.~Rajagopal, P.~Khargonekar, K.~Poolla, and P.~Varaiya, ``Bringing
  wind energy to market,'' \emph{IEEE Transactions on Power Systems}, vol.~27,
  no.~3, pp. 1225--1235, 2012.

\bibitem{Pinson_Trading}
P.~Pinson, C.~Chevallier, and G.~N. Kariniotakis, ``Trading wind generation
  from short-term probabilistic forecasts of wind power,'' \emph{IEEE
  Transactions on Power Systems}, vol.~22, no.~3, pp. 1148--1156, 2007.

\bibitem{wind_forecast}
C.~Monteiro, R.~Bessa, V.~Miranda, A.~Botterud, J.~Wang, G.~Conzelmann
  \emph{et~al.}, ``Wind power forecasting: state-of-the-art 2009.''
  \emph{Technical Report}, 2009, {A}vailable online.

\bibitem{NCAR}
K.~Parks, Y.-H. Wan, G.~Wiener, and Y.~Liu, ``Wind energy forecasting: A
  collaboration of the {National Center for Atmospheric Research (NCAR) and
  Xcel Energy},'' \emph{Technical Report}, October 2011, {A}vailable online.

\bibitem{Forecast_industries}
\emph{Innovations Across The Grid: Partnerships Transforming The Power
  Sector}.\hskip 1em plus 0.5em minus 0.4em\relax The Edison Foundation,
  December 2013, {A}vailable online.

\bibitem{LP_over_distributions}
A.~Shapiro, ``On duality theory of conic linear problems,'' in
  \emph{Semi-infinite programming}.\hskip 1em plus 0.5em minus 0.4em\relax
  Springer, 2001, pp. 135--165.

\bibitem{Boyd_convex}
S.~Boyd and L.~Vandenberghe, \emph{Convex Optimization}.\hskip 1em plus 0.5em
  minus 0.4em\relax Cambridge University Press, 2009.

\bibitem{Han_uncertainty_quantification}
S.~Han, U.~Topcu, M.~Tao, H.~Owhadi, and R.~M. Murray, ``Convex optimal
  uncertainty quantification: Algorithms and a case study in energy storage
  placement for power grids,'' in \emph{Proc. of the American Control
  Conference}, 2013, pp. 1130--1137.

\bibitem{Distributionally_robust}
E.~Delage and Y.~Ye, ``Distributionally robust optimization under moment
  uncertainty with application to data-driven problems,'' \emph{Operations
  Research}, vol.~58, no.~3, pp. 595--612, 2010.

\bibitem{Ambiguous_chance_constrained}
E.~Erdo{\u{g}}an and G.~Iyengar, ``Ambiguous chance constrained problems and
  robust optimization,'' \emph{Mathematical Programming}, vol. 107, no. 1-2,
  pp. 37--61, 2006.

\bibitem{probabilistic_forecast_application}
Z.~Zhou, A.~Botterud, J.~Wang, R.~Bessa, H.~Keko, J.~Sumaili, and V.~Miranda,
  ``Application of probabilistic wind power forecasting in electricity
  markets,'' \emph{Wind Energy}, vol.~16, no.~3, pp. 321--338, 2013.

\bibitem{semi_infinite_prog}
M.~L{\'o}pez and G.~Still, ``Semi-infinite programming,'' \emph{European
  Journal of Operational Research}, vol. 180, no.~2, pp. 491--518, 2007.

\bibitem{Ripley_simulation}
B.~D. Ripley, \emph{Stochastic Simulation}.\hskip 1em plus 0.5em minus
  0.4em\relax John Wiley \& Sons, 2009.

\bibitem{Pinson_scenario}
P.~Pinson, H.~Madsen, H.~A. Nielsen, G.~Papaefthymiou, and B.~Kl{\"o}ckl,
  ``From probabilistic forecasts to statistical scenarios of short-term wind
  power production,'' \emph{Wind Energy}, vol.~12, no.~1, pp. 51--62, 2009.

\end{thebibliography}

%\printbibliography

\end{document}